\documentclass[10pt]{article}

\usepackage{graphicx}
\usepackage{bbding}
\usepackage{perpage}
\usepackage{mathtools}
\usepackage{calc}
\usepackage{enumerate}
\usepackage{amsmath}
\usepackage{amsxtra}
\usepackage[square,numbers]{natbib}
\usepackage{amssymb}
\usepackage{amsfonts}
\usepackage{amsthm}
\usepackage[sans]{dsfont}
\usepackage{color}
\usepackage{hyperref}
\usepackage{textcomp}
\usepackage{mathrsfs}
\usepackage[margin=1in]{geometry}
\usepackage{array}

\newcommand{\xb}{\xi_0}
\newcommand{\Ub}{\overline{U}}
\newcommand{\Vb}{\overline{V}}

\newcommand{\bO}{\mathcal{O}}
\newcommand{\Ps}{\mathcal{P}}
\newcommand{\Peps}{\Ps_{\ueps}}
\newcommand{\ueps}{\epsilon}

\newcommand{\R}{\mathbb{R}}
\newcommand{\C}{\mathbb{C}}

\newcommand{\Linf}{L^\infty}
\newcommand{\BV}{\operatorname{BV}}
\newcommand{\SBV}{\operatorname{SBV}}

\newcommand{\defm}[1]{{\it #1}}

\newcommand{\topref}[2]{\overset{\text{\eqref{#1}}}{#2}}

\theoremstyle{plain}
\newtheorem{lemma}{Lemma}

\newtheorem{definition}{Definition}

\newtheorem{theorem}{Theorem}
\newtheorem{assumption}{Assumption}

\theoremstyle{remark}
\newtheorem{rem}{Remark}

\title{Steady and self-similar solutions of non-strictly hyperbolic systems of conservation laws\footnote{This material is based upon work partially supported by the
National Science Foundation under Grant No.\ NSF DMS-1054115 and by a Sloan Foundation Research Fellowship.}}
\author{Volker Elling and Joseph Roberts}

\begin{document}

\maketitle

\begin{abstract}
We consider solutions of two-dimensional $m \times m$ systems hyperbolic conservation laws that are constant in time and along rays starting at the origin.  The solutions are assumed to be small $L^\infty$ perturbations of a constant state and entropy admissible, and the system is assumed to be non-strictly hyperbolic with eigenvalues of constant multiplicity.  We show that such a solution, initially assumed bounded, must be a special function of bounded variation, and we determine the possible configuration of waves.  As a corollary, we extend some regularity and uniqueness results for some one-dimensional Riemann problems.

\end{abstract}

\section{Introduction}

\subsection{Outline and motivation}

Consider a 1d system of conservation laws
\[{ U_t + f(U)_x = 0 }\] 
augmented with an \defm{entropy inequality}
\[{ \eta(U)_t + \psi(U)_x \leq 0 }\] 
for an \defm{entropy pair} $(\eta,\psi)$ with convex $\eta$.
Consider 
{\it Riemann problem} initial data
\[{ U(x,t=0) = \begin{cases}
        U_- , & x < 0 \\
        U_+ , & x > 0 
\end{cases} }\] 
The solution is expected to be {\it self-similar}:
\[{ U(x,t) = U(\xi) \ ,\ \xi=\frac{x}{t} }\] 
which reduces the system to
\[{ [f(U)-\xi U]_\xi = U \quad,\quad [\psi(U)-\xi\eta(U)]_\xi \leq \eta(U) }\] 
which, for smooth $U$, is equivalent to the non-divergence form
\begin{alignat}{1} [f_U(U)-\xi I]U_\xi = 0 \quad,\quad [\psi_U(U)-\xi\eta_U(U)]U_\xi \leq 0 \label{eq:nondiv} \end{alignat}
In prior work \cite{elling2012steady} we assume:\\
1. $U$ has values in a closed ball $\Peps\subset\Ps$ of small radius $\epsilon$ around $\Ub\in\R^m$. \\
2. The system is \defm{strictly hyperbolic}: for every $U\in\Peps$, $f_U(U)$ has $m$ real and simple eigenvalues $\lambda^\alpha(U)$
with associated right eigenvectors $r^\alpha(U)$. \\
3. Each eigenvalue is either \defm{linearly degenerate},
\begin{alignat}{1} \forall U\in\Peps: \lambda^\alpha_U(U)r^\alpha(U)=0\quad, \label{eq:lindeg} \end{alignat}
or \defm{genuinely nonlinear}:
\begin{alignat*}{1} \forall U\in\Peps: \lambda^\alpha_U(U)r^\alpha(U)\neq0\quad.  \end{alignat*}
Under these conditions we show, among other results, that \\
1. $U\in\SBV\subset\BV$ (as a function of $\xi$), and \\
2. for $\xi\approx\lambda^\alpha$ where $\lambda^\alpha$ is linearly degenerate, $U$ is constant on each side of a \defm{contact discontinuity}.

A key step of the proof was to show the following is impossible: for linearly degenerate $\lambda^\alpha$ and $I$ a closed interval of \emph{positive} length, 
\begin{alignat}{1} \forall \xi\in I: \lambda^\alpha(U(\xi)) = \xi\quad. \label{eq:xilam} \end{alignat}
This is straightforward if we already know $U\in\BV$: 
then $U$ is differentiable in at least one\footnote{in fact almost everywhere} $\xi\in I$, so 
\begin{alignat*}{1} 0 
    &\topref{eq:nondiv}{=} 
    [f_U(U(\xi))-\xi I]U_\xi(\xi) 
    \\&\topref{eq:xilam}{=}
    [f_U(U(\xi))-\lambda^\alpha(U(\xi))I] U_\xi(\xi)
    \\
    &\Rightarrow\quad
    U_\xi(\xi) \parallel r^\alpha(U(\xi))
    \\
    &\topref{eq:lindeg}{\Rightarrow}\quad
    \lambda^\alpha_U(U(\xi))U_\xi(\xi) = 0
\end{alignat*}
But taking $\partial_\xi$ of \eqref{eq:xilam} yields a direct contradiction: 
\[{ \lambda^\alpha_U(U(\xi))U_\xi(\xi) = 1 }\] 
(There is no contradiction if $I$ is a point (so that $\partial_\xi$ cannot be taken), which corresponds to the familiar case of a contact discontinuity.)

The obstacle was to \emph{prove} --- not assume --- that $U\in\BV$, merely starting from $U\in\Linf$ which does not imply differentiability. 
A key ingredient is a result of Saks \cite{saks1937theory}: for an arbitrary \defm{scalar} $g:I\rightarrow\R$ 
there is a sequence $(\xi_n)\subset I$ converging to $\xi_0$ so that 
$U_{|D}$, $U$ restricted to $D=\{\xi_0\}\cup S$, is differentiable in $\xi_0$. 

Important systems like the 2d steady non-isentropic Euler equations 
with supersonic background state $\Ub$ (where one space direction acts as time $t$ and the other as $x$, see Appendix B)
have a linearly degenerate eigenvalue of multiplicity $p\geq2$ (physically corresponding to contact discontinuities that have, 
in addition to tangential velocity jumps, entropy jumps as well). Including passively transported $y,z$ velocities yields even higher
multiplicities $p$.
While many results for strictly hyperbolic systems generalize to $p\geq 2$ in relatively straightforward ways, this is \emph{not} the case here. 
Reason: Saks result is \emph{not} true for $g:I\rightarrow\R^p$ with $p\geq 2$. In this case, \cite{MR2921865}
shows $C^{0,\beta}$ continuity of $U_{|D}$ in $\xi_0$ with \emph{optimal} exponent $\beta=\frac1p<1$. 

\subsection{Related work}

There has been some related work on determining the regularity of weak solutions.   Though $BV$ is the standard setting for one-dimensional systems of conservation laws, a result due to Rauch \cite{MR859822} shows that $BV$ is not appropriate for unsteady multi-dimensional problems. An example of a one-dimensional $SBV$ \footnote{ The Lebesgue decomposition expresses a BV function as the sum of an absolutely continuous function, a singular (differentiable almost everywhere with zero derivative) continuous function (referred to as the \emph{Cantor part}), and a jump (saltus) function.  A special function of bounded variation (SBV) has vanishing Cantor part.} regularity result is \cite{ambrosio2004note}, in which  Ambrosio and De Lellis showed that (not self-similar) $L^\infty$ entropy solutions to \emph{scalar} one-dimensional conservation laws are special functions of bounded variation (as a function of $x$, for all but countably many values of $t$).  In   \cite{MR2433867}, Dafermos showed that $BV$ self-similar solutions to one-dimensional systems with genuinely nonlinear fields are $SBV$ (without reference to entropy). In \cite{bianchini2012sbv}, Bianchini and Caravenna showed that $BV$ entropy solutions to one-dimensional systems with genuinely nonlinear fields are $SBV$ for all but countably many values of $t$.  Though we require self-similarity and a smallness assumption, the fact that we can treat systems with degenerate fields and only assume $L^\infty$ is interesting.

\subsection{Overview}

In this paper we extend the results in \cite{elling2012steady} to non-strictly hyperbolic systems with constant multiplicity.  In Appendix \ref{eul}, we verify that the required assumptions are satisfied for the full Euler equations.  Though the regularity of steady and self-similar solutions of full Euler was proved by the second author in \cite{Roberts:2012fk}, that study was not limited to small perturbations.  Therefore, there was less that could be said regarding the structure of solutions, since strong shocks and subsonic flows were allowed --- there was no notion of ``sectors'' (defined in Section \ref{sect}) in which we know all waves of a given family must occur.  That analysis used the algebraic properties of the full Euler equations instead of the implicit function theorem, and in this paper we treat general systems as well as verify that the full Euler equations fit into this perturbative framework in order to get more structural information than was available in \cite{Roberts:2012fk}.

\section{Physical Systems and Entropy Solutions}
A \emph{system of two-dimensional conservation laws} is a system of nonlinear partial differential equations of the form
\begin{equation} \label{2dlaw}
U_t + f^x(U)_x + f^y(U)_y = 0.
\end{equation}
The unknown $U(t,x,y)  = \big(U^1(t,x,y), U^2(t,x,y), ..., U^m(t,x,y)\big)$ is a function from $\R_+ \times \R^2$ to  $\Ps \subset \R^m$, the individual components $U^\alpha, \alpha = 1,...,m$ are called the \emph{conserved quantities}, the set $\Ps$ is called the \emph{phase space} of physically allowed values, and the smooth functions $f^x, f^y : \Ps \rightarrow \R^m$ are called the \emph{horizontal} and \emph{vertical flux functions}, respectively.  (Like $U$, $f^x$ and $f^y$ are each column vectors with components $f^{x\alpha}, f^{y\alpha}, \alpha = 1,...,m$.)

Smooth functions $\eta, \psi^x, \psi^y : \Ps \rightarrow \R$ are an \emph{entropy-entropy flux pair} for the system \eqref{2dlaw} if, for all $U \in \Ps$,
\begin{align}
\eta_{UU}& \text{ is positive definite}, \label{pd} \\
\psi^x_U &= \eta_U f^x_U, \text{ and} \label{eefh}\\
\psi^y_U &= \eta_U f^y_U. \label{eefv}
\end{align}
\begin{definition}
A \emph{physical} system is a choice of conserved quantities, phase space, flux functions, and entropy-entropy flux pair as described above.
\end{definition}

\begin{definition}
$U \in L^\infty(\R_+ \times \R^2 ; \Ps)$ is a \emph{weak solution} to \eqref{2dlaw} if for any test function $\Phi = (\Phi^1 \, \Phi^2 ... \Phi^m) \in C^\infty_c (R_+ \times \R^2; \R^m)$
\begin{equation}
- \int _{\R_+ \times \R^2} \Phi_t U + \Phi_x f^x(U) + \Phi_y f^y(U) d(t,x,y) = 0. \label{integralform}
\end{equation}
\end{definition}
Unfortunately, weak solutions are not unique for given initial data.  Therefore, we need an \emph{admissibility criterion} for weak solutions.  In this case, the appropriate requirement is that
\begin{equation} \label{ent}
-\int_{\R_+ \times \R^2} \Theta_t \eta(U) + \Theta_x \psi^x(U) + \Theta_y \psi^y(U) \, \, d(t,x,y) \leq 0,
\end{equation}
for any \emph{non-negative} test function $\Theta \in C^\infty_c (\R_+ \times \R^2; \R)$.
This is what is usually referred to as the \emph{integral form} of the differential inequality
\begin{equation}
\eta(U)_t + \psi^x(U)_x + \psi^y(U)_y \leq 0.
\end{equation}

\begin{definition}
$U \in L^\infty(\R_+ \times \R^2 ; \R^m)$ is an \emph{entropy solution} to a physical system \eqref{2dlaw} if it is a weak solution that also satisfies \eqref{ent} for all non-negative test functions $\Theta$.
\end{definition}

\section{Steady and Self-Similar Solutions} \label{section:stss}

We are interested in entropy solutions that are steady in time and constant on rays emanating from the origin.  Using the same arguments as in \cite{elling2012steady}, we have the following.
\begin{definition}
A \emph{steady and self-similar entropy solution} $U \in L^\infty$ to a physical system \eqref{2dlaw} satisfies, in the sense of distributions,
\begin{align} \label{weakformU}
\left\{ \begin{array}{ll} \big( f^y(U) - \xi f^x(U)\big)_\xi + f^x(U) = 0, & \\ \big( \psi^y(U) - \xi \psi^x(U)\big)_\xi + \psi^x(U) \leq 0, & x > 0 \\
 \big( \psi^y(U) - \xi \psi^x(U)\big)_\xi + \psi^x(U) \geq 0, & x < 0 \end{array}. \right.
\end{align}

\medskip

\noindent (Once either the right or left half plane is chosen, $U$ is only a function of $\xi = y/x$.)
\end{definition}

\begin{rem}
We will later justify why we can ignore the case of $x=0$ without loss of generality.
\end{rem}

For the remainder of this chapter, we assume the following.

\begin{assumption}
The system of conservation laws under consideration, \eqref{2dlaw}, is a physical system, and all solutions considered are steady and self-similar entropy solutions.
\end{assumption}

\section{Smallness and Intuition} \label{intuit}
Many of our results require the implicit function theorem, and therefore a smallness assumption.  To that end we assume that our phase space $\Ps$ is a small neighborhood of a constant background state $\Ub$, and rename it $\Peps$.  

\begin{assumption}
The phase space of allowed values for the conserved quantities is of the form
\begin{equation}
\Peps := \Big\{ U \in \R^m \Big| |U-\Ub| \leq \ueps \Big\},
\end{equation}
for some small $\ueps > 0$.  Thus the solutions we consider satisfy
\begin{equation}
||U(\cdot)-\Ub||_{L^\infty} \leq \ueps.
\end{equation}
\end{assumption}    

We will reduce $\ueps$ as necessary throughout, but only finitely many times.

Recall that we are not assuming any regularity of the entropy solution $U(\cdot)$ --- only that it is bounded.  Therefore we cannot expect it a priori to be differentiable anywhere.  We are not assuming it is of bounded variation either, and so we cannot even analyze its derivative in the sense of measures (or even talk about left or right limits).  However, if we look at the differentiated form of the equations anyway, we obtain from the first line of \eqref{weakformU}
\begin{equation}
\big( f^y_U - \xi f^x_U \big)U_\xi = 0.
\end{equation}
Therefore, if on some interval $\xi$ is not a generalized eigenvalue of the matrix pair $\Big(f^x_U\big(U(\xi)\big),f^y_U\big(U(\xi)\big)\Big)$, $U_\xi = 0$ and therefore $U$ is constant.  If instead there is some interval of $\xi$ on which $\xi$ is an eigenvalue, then $U_\xi$ must be the associated eigenvector, which is precisely the case of a \emph{simple wave} for the steady problem.  Either way, the smallness assumption on the phase space suggests that when $\xi$ is not close to a generalized eigenvalue of the matrix pair $\big(f^x_U(\Ub),f^y_U(\Ub)\big)$, $U$ must be constant.

Shocks and contact discontinuities are also possible, but the smallness assumption and standard facts about conservation laws suggests that these waves also occur near the eigenvalues evaluated at the background state $\Ub$, since the phase space is a  small neighborhood of $\Ub$.  

All together, the differentiated form suggests that all interesting behavior must occur near the eigenvalues of the background state.  So even if we were allowed to use the differentiated form it would be important to analyze the characteristic behavior of our system, which is the next step.

\section{Pointwise Information}
It is cumbersome to work with the integral form of a conservation law, so we instead investigate what pointwise information we can derive from it.  Recall that we are assuming that $U \in L^\infty$.  Rearranging the first line of \eqref{weakformU} yields
\begin{equation}
\big(f^y(U)-\xi f^x(U)\big)_\xi = -f^x(U).
\end{equation}
Therefore, the quantity
\begin{equation}
\big(f^y(U)-\xi f^x(U)\big)
\end{equation}
has a distributional derivative that is $L^\infty$ (since $f^x$ is smooth on $\Ps$), and is therefore almost everywhere equal to a Lipschitz continuous function.  Therefore, the fundamental theorem of calculus asserts that

\begin{equation}
\begin{split}
\Big(f^y\big(U(\xi_2)\big)-\xi_2 f^x\big(U(\xi_2)\big)\Big)&-\Big(f^y\big(U(\xi_1)\big)-\xi_1 f^x\big(U(\xi_1)\big)\Big) \\ &= -\int_{\xi_1}^{\xi_2} f^x\big(U(\eta)\big) \, \, d\eta , \text{ for a.e. } \xi_1, \xi_2. 
\end{split}
\label{pointlaw}
\end{equation}

Similarly, the second line of \eqref{weakformU} shows that the distributional derivative of
\begin{equation}\label{nonposdist}
\Big(\psi^y\big(U(\xi)\big) - \xi \psi^x\big(U(\xi)\big)\Big) + \int_{\overline{\xi}}^\xi \psi^x\big(U(\eta)\big) \, \, d\eta
\end{equation}
is a non-positive distribution, and therefore a non-positive measure.  Therefore, there is a version of \eqref{nonposdist} that is a non-increasing function of bounded variation, and rearranging this statement we obtain

\begin{equation}
\begin{split}
\Big(\psi^y\big(U(\xi_2)\big)&-\xi_2 \psi^x\big(U(\xi_2)\big)\Big)-\Big(\psi^y\big(U(\xi_1)\big)-\xi_1 \psi^x\big(U(\xi_1)\big)\Big) \\ &\leq -\int_{\xi_1}^{\xi_2} \psi^x\big(U(\eta)\big) \, \, d\eta , \text{ for } x > 0 \text{ and a.e. } \xi_1 < \xi_2 .
\end{split}
\label{pointentf}
\end{equation}

Similarly, we can obtain
\begin{equation}
\begin{split}
\Big(\psi^y\big(U(\xi_2)\big)&-\xi_2 \psi^x\big(U(\xi_2)\big)\Big)-\Big(\psi^y\big(U(\xi_1)\big)-\xi_1 \psi^x\big(U(\xi_1)\big)\Big) \\ &\geq -\int_{\xi_1}^{\xi_2} \psi^x\big(U(\eta)\big) \, \, d\eta , \text{ for } x < 0 \text{ and a.e. } \xi_1 < \xi_2 .
\end{split}
\label{pointentb}
\end{equation}
We now fix a version of $U$ so that these equations and inequalities hold \emph{for all} $\xi_1 < \xi_2$, as discussed in Section 7 of \cite{elling2012steady}, so that

\begin{equation} \label{everywherepoint}
\left\{ \begin{array}{*2{>{\displaystyle}l}}
\big(f^y(U)-\xi f^x(U)\big) \Big|_{\xi_1}^{\xi_2} = -\int_{\xi_1}^{\xi_2} f^x\big(U(\eta)\big) \, \, d\eta , & \\[10pt] 
\big(\psi^y(U)-\xi \psi^x(U)\big) \Big|_{\xi_1}^{\xi_2} \leq -\int_{\xi_1}^{\xi_2} \psi^x\big(U(\eta)\big) \, \, d\eta, & x > 0, \\[10pt]
\big(\psi^y(U)-\xi \psi^x(U)\big) \Big|_{\xi_1}^{\xi_2} \geq -\int_{\xi_1}^{\xi_2} \psi^x\big(U(\eta)\big) \, \, d\eta, & x < 0,\end{array} \right.
\end{equation}
for all $\xi_1 < \xi_2$.

\section{Hyperbolicity}
Consider the homogeneous polynomial
\begin{equation}
P(x:y) := \det \big(\vec{x} \times \vec{f}_U(\Ub)\big) = \det \big(xf^y_U(\Ub)-yf^x_U(\Ub)\big),
\end{equation}
where $(x:y)$ are homogeneous coordinates on $\R \mathbb{P}^1$, $\vec{x} = (x,y)$, and $\vec{f} = (f^x, f^y)$.   We call the background state $\Ub$ \emph{hyperbolic} if $P(x:y)$ has exactly $m$ roots in $\R \mathbb{P}^1$, counting multiplicity.  If $R$ is any rotation matrix, it is easy to see that $R\vec{x} \times R\vec{f}_U$ = $\vec{x} \times \vec{f}_U$.  Therefore, if $\vec{f}$ is replaced by $R\vec{f}$, the roots of $P$ are rotated by the same amount.  $P$ has degree at most $m$, so there are at most $m$ distinct roots and therefore we can rotate $\vec{f}$ to ensure that $(0:1)$ is not a root of $P$.  Assume without loss of generality that this has been done.

The fact that $(0:1)$ is not a root of $P$ immediately implies that 
\begin{equation}
\det f^x_U(\Ub) \neq 0.
\end{equation}
Now consider the polynomial
\begin{equation}\label{det2}
p(\xi) := \det\big(f^y_U(\Ub)-\xi f^x_U(\Ub)\big).
\end{equation}
From the above discussion, $p$ has $m$ real roots, counting multiplicity, since each of the $m$ roots of $P$ lead to a finite (since none lie on the $y$-axis) value of $\xi$ that is a root of $p$.
A \emph{generalized eigenvalue} $\lambda(U)$ and an associated \emph{generalized eigenvector} $r(U)$ of the matrix pair $\big(f^x_U(U), f^y_U(U)\big)$ satisfy
\begin{equation}
\big(f^y_U(U) - \lambda(U) f^x_U(U)\big)r(U) = 0.
\end{equation}
(Note that the term \textit{generalized eigenvectors} in this context refers to the elements of the kernel of the linear matrix pencil $(-\lambda f^x_U + f^y_U)$, not elements of the kernel of $(-\lambda f^x_U + f^y_U)^k$ for $k >1$ in the context of defective geometric multiplicity.) 

For the remainder, we shall assume that the steady problem is hyperbolic in the following sense.
\begin{definition}
The steady problem associated to the system \eqref{2dlaw}  is called \emph{hyperbolic} on the phase space $\Peps$ if the generalized eigenvalues of the matrix pair $\big(f^x_U(U),f^y_U(U)\big)$ are real, semisimple, and of constant multiplicity on $\Peps$ (thus the sum of the multiplicities equals $m$).  It is called \emph{strictly hyperbolic} if there are $m$ distinct generalized eigenvalues, and \emph{non-strictly hyperbolic with constant multiplicity} otherwise.
\end{definition}

\begin{assumption} \label{hyper}
The steady problem associated to \eqref{2dlaw} is non-strictly hyperbolic with constant multiplicity on the phase space $\Peps$ (which implies $f^x_U(U)$ is non-degenerate for all $U \in \Peps$).
\end{assumption}

As discussed before, the strictly hyperbolic case was treated in \cite{elling2012steady}, and the purpose of this paper is to treat repeated eigenvalues.

\section{Change of Dependent Variables}
In light of Assumption \ref{hyper}, we can perform a change of dependent variables as in \cite{elling2012steady}.  Since $f^x_U(\Ub)$ is non-degenerate, 
\begin{equation}
U \mapsto V := f^x(U)
\end{equation}
is a diffeomorphism on $\Peps$, after reducing $\epsilon$ if necessary.  We then define
\begin{align}
\Vb := f^x(\Ub), \qquad & \qquad f(V) := f^y\big(U(V)\big), \\
e(V) := \psi^x\big(U(V)\big), \qquad & \qquad q(V) := \psi^y\big(U(V)\big).
\end{align}
Then,
\begin{align}
f_V = f^y_U U_V.
\end{align}
Also, we have
\begin{align}\label{evetau}
e_V = \psi^x_U U_V = \eta_U f^x_U U_V = \eta_U,
\end{align}
and
\begin{align}
q_V = \psi^y_U U_V = \eta_U f^y_U U_V = e_V f_V.
\end{align}
Therefore, $e$ and $q$ form an ``entropy-entropy flux pair'' for the flux $f$.  The term is applied loosely here because $e$ is not necessarily convex.  Properties of the entropy are only needed in two instances, and further properties of $e$ will be discussed when they are needed.

Abusing notation, we shall continue to refer to our phase space as $\Peps$, but it will now refer to a small ball around $V$ of permissible values.

The differential equations and entropy inequalities for $V$ are then
\begin{equation} \label{weakformV}
\left\{ \begin{array}{ll} \big( f(V)- \xi V\big)_\xi + V = 0, & \\ \big( q(V) - \xi e(V)\big)_\xi + e(V) \leq 0, & x > 0 \\
 \big( q(V) - \xi e(V) \big)_\xi + e(V) \geq 0, & x < 0 \end{array}. \right.
\end{equation}

The pointwise form is
\begin{equation} \label{everywherepointV}
\left\{ \begin{array}{*2{>{\displaystyle}l}}
\big(f(V)-\xi V\big) \Big|_{\xi_1}^{\xi_2} = -\int_{\xi_1}^{\xi_2} V(\eta) \, \, d\eta , & \\[10pt] 
\big(q(V)-\xi e(V)\big) \Big|_{\xi_1}^{\xi_2} \leq -\int_{\xi_1}^{\xi_2} e\big(V(\eta)\big) \, \, d\eta, & x > 0, \\[10pt]
\big(q(V)-\xi e(V)\big) \Big|_{\xi_1}^{\xi_2} \geq -\int_{\xi_1}^{\xi_2} e\big(V(\eta)\big) \, \, d\eta, & x < 0,\end{array} \right.
\end{equation}
for all $\xi_1 < \xi_2$.

\section{Eigenvalues and Eigenvectors}
Since $f_V = f^y_U U_V = f^y_U (f^x_U)^{-1}$, 
\begin{equation}
\det ( f^y_U - \lambda f^x_U) = 0 \iff \big(\det ( f^y_U - \lambda f^x_U)\big)\big( \det (f^x_U)^{-1}) = 0 \iff \det(f_V - \lambda I) = 0.
\end{equation}
Therefore, the generalized eigenvalues of the matrix pair $(f^x_U,f^y_U)$ are precisely the eigenvalues of the matrix $f_V$.

To that end, define
\begin{equation}
\lambda^1(V)< \lambda^2(V)< \cdots < \lambda^n(V)
\end{equation}
to be the distinct values of $\lambda$ solving
\begin{equation}
\det \big(f_V(V)-\lambda I )\big) = 0.
\end{equation}
If $s$ is a generalized eigenvector with eigenvalue $\lambda$, then
\begin{equation}
0 = (f^y_U-\lambda f^x_U )s =( f^y_U-\lambda f^x_U)(f^x_U)^{-1}f^x_Us=(f_V-\lambda I )f^x_U s,
\end{equation}
which means that $f^x_U s$ is an eigenvector of $f_V$.  Since the generalized eigenvectors span $\R^m$, the eigenvectors of $f_V$ do as well.  Define
\begin{align}
R^\alpha(V) := \mbox{ker }\big(f_V(V)-\lambda^\alpha(V)I\big) , \qquad p_\alpha(V) := \mbox{dim }R^\alpha(V),
\end{align}
so that, under the hyperbolicity assumption 
\begin{equation}
p_\alpha(V) \equiv : p_\alpha \textrm{ on } \Peps,
\end{equation}
and
\begin{equation}
\R^m = \bigoplus_{\alpha=1}^n R^{\alpha}(V) \textrm{ for all } V \in \Peps.
\end{equation}

In the strictly hyperbolic setting, it is relatively easy to prove that if the matrix $f_V$ is a smooth function of $V$, then so are the eigenvalues and eigenvectors (see for example \cite{MR2597943}).  However, the situation is more delicate in the case of repeated eigenvalues --- there are examples in which the eigenvalues and eigenvectors are not as smooth as the matrix.  Fortunately, if the eigenvalues are semisimple and of constant multiplicity, they and the eigenvectors can be shown to be as smooth as the matrix, though the individual eigenvectors are only guaranteed to be locally defined as smooth functions (the \emph{eigenspaces} are smooth when given a suitable topology).  Since we are only considering small perturbations, we can simply reduce $\epsilon$ if necessary and have our right and left eigenvectors defined on all of $\Peps$.

The smoothness of the fluxes, the hyperbolicity assumption, and the discussion in Section \ref{evaluesmooth} in the Appendix allow us to conclude that, for each $\alpha = 1,.., n$,
\begin{equation}
\lambda^\alpha : \Peps \rightarrow \R
\end{equation}
is smooth.  In addition, we have for all $V \in \Peps$ an orthonormal basis for $R^\alpha(V)$:
\begin{equation}
R^\alpha(V) = \textrm{Span } \big\{ r^{\alpha,1}(V), ..., r^{\alpha, p_\alpha}(V) \big\}.
\end{equation}
Reducing $\epsilon$ as necessary, we have that the right and left  eigenvectors 
\begin{equation}
r^{\alpha,i}(V),l^{\alpha,i}(V) : \Peps \rightarrow \R^m
\end{equation}
are smooth, and satisfy the normalization
\begin{align}
|r^{\alpha,i}(V)| &= 1,\\
l^{\alpha,i}(V)r^{\beta,j}(V)&=\delta_{\alpha \beta}\delta_{ij} \quad \forall \alpha, \beta = 1,..,n, i = 1,..,p_\alpha, j = 1,..,p_\beta.
\end{align}

If for some $\alpha$, $p_\alpha = 1$, then we omit the second index of the eigenvector and simply denote it by $r^\alpha$.

\section{Genuine Nonlinearity and Linear Degeneracy, Convexity of $e$} \label{gennonlindeg}
The results in this paper are systems with eigenvalues that are either linearly degenerate on all of $\Peps$ or genuinely nonlinear on all of $\Peps$.  
\begin{definition}
An eigenvalue $\lambda^\alpha$ is \emph{linearly degenerate} if 
\begin{equation}
\lambda_V(V) r^{\alpha, i}(V) \equiv 0 \textrm{ on } \Peps, \quad i = 1,..,p_\alpha.
\end{equation}
\end{definition}
As it turns out, if $p_\alpha \geq 2$, $\lambda^\alpha$ must be linearly degenerate, and there is a nice geometrical structure created by the eigenspaces.  This result is originally due to Boillat, and this theorem and its proof can be found in \cite{serre1999systems}.

\begin{theorem} [Boillat, as in \cite{serre1999systems}] \label{boillat} Suppose the hyperbolicity assumption (Assumption \ref{hyper}) is satisfied.  If an eigenvalue $\lambda^\alpha$ has multiplicity $p_\alpha \geq 2$, then it must be linearly degenerate.  In addition, the affine subspaces $V + R^\alpha(V)$ are the tangent spaces to a family of sub-manifolds of dimension $p_\alpha$. Each integral submanifold can be parameterized by $s \rightarrow W^\alpha(V^-, s),$ with $s$ in $\R^{p_\alpha}$, so that
\begin{align}
W^{\alpha}(V^-,0) &= V^- \\
\mbox{\emph{Span} }\Big\{W^\alpha_{s^i}(V^-,s)\Big\}_{i=1}^{p_\alpha} &= R^{\alpha}(V^-). \label{eq:charsub}
\end{align}

They form a foliation of $\Peps$ called the characteristic foliation associated with $\lambda^{\alpha}$.  
\end{theorem}

We now recall the definition of a genuinely nonlinear eigenvalue, which by the previous theorem must have multiplicity 1.
\begin{definition}
An eigenvalue $\lambda^\alpha$ is \emph{genuinely nonlinear} if
\begin{equation}
\lambda_V(V)r^\alpha(V) \neq 0 \textrm{ on } \Peps.
\end{equation}
\end{definition}
We can orient $r^\alpha(V)$ so that, without loss of generality,
\begin{equation}
\lambda_V(V)r^\alpha(V) > 0 \textrm{ on } \Peps.
\end{equation}
We make the following assumption.
\begin{assumption}
Each simple eigenvalue is either genuinely nonlinear or linearly degenerate.  $($Recall that an eigenvalue with multiplicity greater than one \emph{must be} linearly degenerate, so no assumption is necessary in that case.$)$
\end{assumption}

We will, as in \cite{elling2012steady}, need to construct an averaged matrix to continue the analysis.  To ensure the existence of an appropriate averaged matrix, we need the ``entropy'' $e$ to be convex (or concave so that $-e$ is convex).  Since $e$ is the original horizontal entropy flux composed with the inverse of the horizontal flux, and it is the original entropy $\eta$ that is convex, we must carefully inspect whether $e$ is convex or not.

\begin{lemma} \label{econvex}
If $f^x(U)$ has only positive $($negative$)$ eigenvalues, then $e$ is strictly convex $($concave$)$.
\end{lemma}

\begin{proof}
We shall use Proposition 6.1 from \cite{serre2010matrices}.  It states that if $H$ is symmetric positive definite, and $K$ is symmetric, then $HK$ is diagonalizable with real eigenvalues.  Moreover, the number of positive (negative) eigenvalues of $K$ equals the number of positive (negative) eigenvalues of $HK$.  First, from \eqref{evetau}, we have that
\begin{align}
e_{VV}=\eta_{UU}U_V,
\end{align}
which we rewrite as
\begin{align}
(f^x_U)^{-1} = (\eta_{UU})^{-1} e_{VV}. 
\end{align}
By assumption, $(\eta_{UU})^{-1}$ is symmetric positive definite, and $e_{VV}$ is symmetric.  Then, applying the proposition, if all the eigenvalues of $(f^x_U)^{-1}$ are positive (negative), then $e_{VV}$ is positive (negative) definite, since a symmetric matrix is positive (negative) definite if and only if its eigenvalues are all positive (negative).  \end{proof}

We finally note that changing variables to $V$ does not affect linear degeneracy or genuine nonlinearity.  As we showed before, if $s$ is a generalized eigenvector of the matrix pair $(f^x_U,f^y_U)$, then $r = f^x_U s$ is an eigenvector of $f_V$, and so
\begin{equation}
\lambda_V r = \lambda_V (f^x_U s) = \lambda_U U_V V_U s = \lambda_U s.
\end{equation}
This does not matter much when discussing general systems, but when checking for linear degeneracy or genuine nonlinearity in a specific system it is more natural to check in terms of the original conserved quantities $U$ or perhaps another choice of state variables.

\section{Averaged Matrix}\label{matravg}

To proceed with the analysis, we need to construct an averaged matrix $\hat{A}$ that is
\begin{itemize}
\item smooth and diagonalizable (with real eigenvalues),
\item satisfies $\hat{A}(V^-,V^+)(V^+-V^-) = f(V^+)-f(V^-)$, and
\item satisfies $\hat{A}(V,V) = f_V(V)$.
\end{itemize}
A commonly used choice that suffices in \cite{elling2012steady} is to use
\begin{equation} \label{naive}
\hat{A}(V^-,V^+) := \int_0^1 f_V\big(V^- + s(V^+-V^-)\big)ds.
\end{equation}
Clearly this satisfies most of the requirements --- but in fact this definition only guarantees a diagonalizable $\hat A$ in the \emph{strictly hyperbolic} case.  This is because the set of matrices with all simple real eigenvalues is open, and we are only interested in $V^\pm \in \Peps$, and so smoothness of the flux guarantees that this averaged matrix is a small perturbation of $f_V(\Vb)$.  However, the set of matrices with repeated real eigenvalues is \emph{not open}, and so in the repeated eigenvalue case this $\hat{A}$ is not guaranteed to be diagonalizable.  

An averaged matrix of this type is often used in numerical computations, and for some specific systems there is a \emph{Roe averaged matrix} available.  It has the special property that it can be defined as
\begin{equation}
\hat{A}(V^-,V^+) = f_V(\hat{V}),
\end{equation}
where $\hat{V}$ is some appropriate averaging of the states $V^-$ and $V^+$.  In this case, diagonalizability is guaranteed from diagonalizability of $f_V$.  (It also has the useful property that expressions for the eigenvalues and eigenvectors of $f_V$ are available, which makes it especially well suited for numerical computations.)  This is often accomplished by doing a line integral in phase space between the two states, but choosing a more sophisticated path than simply the line segment between the two states which allows one to analytically evaluate the integral.  However this is not available for general systems since it is reliant on the algebraic structure of the flux functions, and so we need another approach.

A theorem due to Harten and Lax that appeared in \cite{MR694161} shows that physical systems always possess an averaged matrix with these properties.  Therefore, we need either $e_{VV}$ to be positive definite, or negative definite so that $(-e)$ can function as a convex entropy.   In light of Lemma \ref{econvex}, we will need to assume that the eigenvalues of $f^x_U(\Ub)$ are all the same sign (from which it follows that they will all have the same sign on all of $\Peps$).

\begin{theorem}[Harten, Lax as in \cite{MR694161}]   \label{HL}
Suppose there exists an entropy/entropy-flux pair $(e, q)$ for the flux function $f$ with $e_{VV}$ positive definite.  Then we can define an averaged matrix $\hat{A}(V^\pm)$, such that it is smooth in $V^\pm$, $\hat{A}(V,V)=f_V(V)$, and it is diagonalizable with real eigenvalues for all $V^\pm \in \Peps$.  Most importantly,
\begin{align}
f(V^+)-f(V^-) = \hat{A}(V^-,V^+)(V^+-V^-).
\end{align}
\end{theorem}
A key idea in the proof is to use a real symmetric square root of the matrix $e_{VV}$, which is only possible if $e_{VV}$ is symmetric positive definite.
Since we are assuming that at least one eigenvalue has multiplicity greater than one, we need to make the following assumption in order to have an averaged matrix.
\begin{assumption} \label{avgmxexists}
Assume that the eigenvalues of $f^x_U(\Ub)$ are all positive or all negative.
\end{assumption}

We denote the eigenvalues of $\hat{A}(V^\pm)$ as $\hat{\lambda}^{\alpha, i}.$ If for some $\alpha$, $p_\alpha=1$, then from the discussion in Section \ref{evaluesmooth} in the Appendix it follows that $\hat{\lambda}^\alpha$ and $\hat{r}^\alpha$ are smooth functions of $V^\pm$.  If instead $p_\alpha > 1$, we will not have in general that $\hat{\lambda}^{\alpha,1}=...=\hat{\lambda}^{\alpha,p_\alpha}$.   Since the multiplicity of these eigenvalues is not necessarily constant on $\Peps \times \Peps$, we cannot conclude that these eigenvalues and their associated eigenvectors are smooth functions of $V^\pm$.

However, as discussed in Section \ref{evaluesmooth} in the Appendix, the eigenvalues are continuous functions of $V^\pm$.  We define, for $\alpha = 1,..,n$, the $\alpha$-\textit{group} to be 
\begin{align}
\left\{ \hat{\lambda}^{\alpha,1}, \hat{\lambda}^{\alpha,2},..,\hat{\lambda}^{\alpha,p_\alpha} \right\},
\end{align}
which can be continuously labeled so that 
\begin{align}
\hat{\lambda}^{\alpha,1}(V,V) = \hat{\lambda}^{\alpha,2}(V,V)=...=\hat{\lambda}^{\alpha,p_\alpha}(V,V) = \lambda^{\alpha}(V)
\end{align}
for all $V \in \Peps$.

Many examples exist that demonstrate lack of continuity of eigenvectors when the multiplicity of an eigenvalue changes, even for symmetric matrices.  Moreover, the projection operators onto the eigenspaces also can display this behavior --- it is worse than just not being able to smoothly pick bases for these eigenspaces.  However, if we instead consider, as in \cite{kato1995perturbation}, the \textit{total projection} for the $\alpha-$group, then this will be as smooth as $\hat{A}$.  We have the formula
\begin{align}
\hat{P}_\Gamma (V^\pm) = \frac{1}{2 \pi i} \int_\Gamma (zI-\hat{A}(V^\pm))^{-1} dz,
\end{align}
where $P_\Gamma$ is the sum of the projections onto the eigenspaces of all eigenvalues inside some counterclockwise contour $\Gamma$.  By continuity, all eigenvalues in the $\alpha$-group remain close to $\lambda^\alpha(\Vb)$ for all $V^\pm$ in $\Peps$, and so we have for $\alpha=1,..,n$ that
\begin{align}
\hat{P}^\alpha(V^\pm) = \frac{1}{2 \pi i} \int_{|z-\lambda^{\alpha}(\Vb)| = \delta} (zI-\hat{A}(V^\pm))^{-1} dz
\end{align}
is the total projection for the $\alpha$-group, where $\delta$ is small enough so these curves remain distinct for different choices of $\alpha$, and large enough so as to include the entire $\alpha$-group for all $V^\pm \in \Peps$.
Clearly such a $\delta > 0$ can be found for $\epsilon$ sufficiently small.  
We will make use of the following properties of the projections.
\begin{align}
\hat{P}^\alpha \hat{P}^\beta &= \delta_{\alpha \beta} \hat{P}^\alpha,\\
\hat{P}^\alpha \hat{A} &= \hat{A} \hat{P}^\alpha = \hat{P}^\alpha \hat{A} \hat{P}^\alpha,\\
\sum_{\alpha=1}^n \hat{P}^\alpha &=I.
\end{align}
Note that this implies that $\hat{P}^\alpha \R^m$ is an invariant subspace for $\hat{A}$.

If for some $\alpha$, $p_\alpha = 1$, then $\hat \lambda^{\alpha}$ and its associated right eigenvector $\hat r^\alpha$ are smooth functions of $V^\pm$.  In this case we also normalize so that
\begin{align}
\hat r^\alpha(V,V) &= r^\alpha(V),\\
|\hat r^\alpha(V^\pm)| &= 1, \textrm{ for all } \alpha \textrm{ with } p_\alpha = 1.
\end{align}

Under the assumption \ref{avgmxexists}, we can rewrite the pointwise version as (after slight rearrangement)

\begin{equation} \label{everywherepointavg}
\left\{ \begin{array}{*2{>{\displaystyle}l}}
\begin{split} \Big(\hat{A}\big(V(\xi_1),V(\xi_2)\big)-\xi_1 I\Big)&\big(V(\xi_2)-V(\xi_1)\big) \\ &= \int_{\xi_1}^{\xi_2} V(\xi_2)- V(\eta) \, \, d\eta ,  \end{split} & \\[10pt] 
\big(q(V)-\xi e(V)\big) \Big|_{\xi_1}^{\xi_2} \leq -\int_{\xi_1}^{\xi_2} e\big(V(\eta)\big) \, \, d\eta, & x > 0, \\[10pt]
\big(q(V)-\xi e(V)\big) \Big|_{\xi_1}^{\xi_2} \geq -\int_{\xi_1}^{\xi_2} e\big(V(\eta)\big) \, \, d\eta, & x < 0,\end{array} \right.
\end{equation}
for all $\xi_1 < \xi_2$.

\section{Left and Right Sequences}
We now continue as in \cite{elling2012steady}.  Since $V(\cdot)$ is only assumed to be bounded, it is not necessarily piecewise smooth, and and at each point it is not necessarily either continuous, approximately continuous, or discontinuous with an approximate jump discontinuity (these are the hypotheses typically used in deriving the familiar Rankine-Hugoniot jump conditions).  Arbitrary bounded functions do not even possess well defined limits from the left and right, and so we have to be more careful in this $L^\infty$ setting.

We wish to use the pointwise form \eqref{everywherepointavg}, so to that end we define pairs of sequences $(\tilde \xi_k^-), (\tilde \xi_k^+)$ such that $\tilde \xi_k^- < \tilde \xi_k^+$ for all $k$ and $\tilde \xi_k^\pm \rightarrow \xi$.  

Since $\Peps$ is compact, there exists a subsequence $\tilde \xi^+_{j(k)}$ such that $V(\tilde \xi^+_{j(k)}) \rightarrow V^+$.  Similarly there exists a subsequence $\tilde \xi^-_{i(j(k))}$ such that $V(\tilde \xi^-_{i(j(k))}) \rightarrow V^-$.  By construction, taking $\xi^\pm_k := \tilde \xi^\pm_{i(j(k))}$ yields a pair of sequences such that
\begin{equation}
\xi^\pm_k \rightarrow \xi, \qquad V(\xi^\pm_k) \rightarrow V^\pm, \qquad \xi^-_k < \xi^+_k \textrm{ for all } k.
\end{equation}
Understanding that this depends on the pair of sequences chosen, we define
\begin{equation}
[g(V)] := g(V^+)-g(V^-)
\end{equation}
for any function $g$ of $V$.  Define
\begin{equation}
J(g(V);\xi) := \sup |[g(V)]|,
\end{equation}
where the supremum is taken over all such sequences.  Obviously $J(g(V);\xi) = 0 \iff g \circ V$ is continuous at $\xi$.  

Applying \eqref{everywherepointavg} with $\xi_1 = \xi_k^-$ and  $\xi_2 = \xi_k^+$ (this is why we insisted $\xi^-_k < \xi^+_k$ for all $k$) and taking the limit $k \rightarrow \infty$ yields
\begin{equation} \label{lrseqlim}
\left\{ \begin{array}{ll}
\big(\hat{A}(V^-,V^+)-\xi I \big)[V] = 0, & \\
 \left[q(V)\right]-\xi [e(V)] \leq 0, & x > 0, \\
 \left[q(V)\right]-\xi [e(V)] \geq 0, & x < 0,\end{array} \right.
\end{equation}
with the first line being equivalent to
\begin{equation}
[f(V)] - \xi[V] = 0,
\end{equation}
which is the familiar \emph{Rankine-Hugoniot condition}.  Therefore, even in the absence of well defined left and right limits, we can make some sense of limits using these pairs of sequences, and the usual conditions apply to these sequence-dependent $V^\pm$.  

The first line of \eqref{everywherepointavg} implies that, for a given pair of sequences
\begin{equation} \label{RH}
[V] = 0 \qquad \textrm{ or } \qquad \xi = \hat{\lambda}^{\alpha,i} \mbox{ and } [V] \in \ker (\hat{A}(V^\pm) - \hat{\lambda}^{\alpha,i}I).
\end{equation}

\section{Sectors}\label{sect}
We now start to make rigorous statements similar to the intuition developed in Section \ref{intuit}, and present two theorems from \cite{elling2012steady} relying on the existence of $\hat{A}$.
\begin{theorem}[Elling, Roberts \cite{elling2012steady}]
   \label{th:Vconst}
    Suppose $V$ is continuous on an interval $I=]\xi_1,\xi_2[$ and that
    $\xi$ is not an eigenvalue of $f_V\big(V(\xi)\big)$
    for any $\xi\in I$.
    Then $V$ is constant on $I$.
\end{theorem}

The continuity assumption is stronger than we need, so we proceed without assuming continuity but instead assuming $\xi$ is bounded uniformly away from all eigenvalues of $f_V\big(V(\xi)\big)$.  

\begin{theorem}[Elling, Roberts \cite{elling2012steady}]
    \label{th:Uconst2}%
    Consider an interval $I=\, \, ]\xi_1,\xi_2[$.
    There is a $\delta_s=\delta_s(\ueps)>0$, with
    \begin{equation}
     \delta_s\downarrow 0\quad\text{as}\quad\ueps\downarrow 0,
     \end{equation}
    so that 
    \begin{align}
        \forall\alpha =1,..n, \forall \xi \in I: |\lambda^\alpha(\Vb)-\xi| > \delta_s \label{eq:lamdel}
    \end{align}
    implies $V$ is constant on $I$.
\end{theorem}

Recall that under Assumption \ref{hyper}
\begin{equation}
P(0:1) = \det f^x_U(\Ub) \neq 0.
\end{equation}
Find some $\xi$ such that 
\begin{equation}
P(1:\xi) = \det (f^y_U(\Ub)-\xi f^x_U(\Ub)) \neq 0.
\end{equation}
Rotate coordinates so that $(1:\xi)$ becomes $(0:1)$ and $(0:1)$ becomes $(-1,\xi)$.  In these new coordinates, Assumption \ref{hyper} is satisfied, and so all results apply.  Since $P$ is invariant under rotation, in these new coordinates
\begin{equation}
P(-1, \xi) \neq 0,
\end{equation}
which means that $-\xi \neq \lambda^\alpha(\Vb)$ for all $\alpha = 1,.., n$.  Reduce $\epsilon$ so that $|-\xi - \lambda^\alpha(\Vb)| > \delta_s$ for all $\alpha$.  This will still be true on some small interval around $-\xi$, and so the previous theorem applies.  Rotate back to the original coordinates to deduce that $V$ must be constant in thin sectors containing the positive and negative $y$-axes.  Therefore, we lost no generality in only considering test functions supported away from $x=0$ and doing all the analysis in terms of $\xi$.

Now, construct $n$ intervals of the form
\begin{equation}
I^\alpha := \quad ]\lambda^\alpha(\Vb)-\delta_s, \lambda^\alpha(\Vb)+\delta_s[.
\end{equation}

Considering both $x>0$ and $x<0$, we now have $2n$ thin sectors centered at $\dfrac{y}{x} = \lambda^\alpha(\Vb)$ for some $\alpha$.  $V$ is constant outside these sectors, and we have 
\begin{equation}
\lim_{\xi \rightarrow \pm \infty \atop x > 0} V(\xi), \, \,  = \lim_{\xi \rightarrow \mp \infty \atop x<0} V(\xi).
\end{equation}

All interesting behavior occurs within these sectors, so we will consider the behavior of $V$ in these sectors individually.

\section{Linearly Degenerate Sectors}
\subsection{Satisfying the Jump Conditions}
We now analyze the possible behavior of $V$ in $I^\alpha$, where $\lambda^\alpha$ is linearly degenerate.

The first thing we recall is that if $V$ is discontinuous at $\xi$, then we can find a pair of subsequences whose limits yield
\begin{equation}
[f(V)] - \xi [V] = 0, \label{RH2}
\end{equation}
with $[V] \neq 0$.  We need to show that \eqref{RH2} is satisfied if and only if $V^\pm$ both lie on the same leaf of the foliation discussed earlier.

First, we notice that the eigenvalue $\lambda^\alpha$ is constant on each leaf of the foliation.  By direct calculation, we see that
\begin{align}
\frac{\partial}{\partial s^i} \Big(\lambda^{\alpha}(W^\alpha(V^-,s)) \Big)= \lambda_V^{\alpha}(W^\alpha(V^-,s)) W^\alpha_{s^i}(V^-,s)\equiv 0 \label{lambda-const-manfld}
\end{align}

for all $i \in \{1, ..., p_\alpha\}$ by linear degeneracy (s§ince $W^\alpha_{s^i} \in R^\alpha)$.  Therefore,
\begin{align}
s \mapsto \lambda^\alpha(W^\alpha(V^-,s)) \quad \mbox{is constant.}
\end{align}

We claim that for all $s$, if $V^+ = W^\alpha(V^-,s)$ and $\xi = \lambda^\alpha(V^-) =  \lambda^\alpha(V^+)$ then \eqref{RH2} is satisfied. Make these choices for $V^+$ and $\xi$ and consider
\begin{align}
F(s) := f\big(W^\alpha(V^-,s)\big)-f(V^-) - \lambda^\alpha\big(W^\alpha(V^-,s)\big)\big(W^\alpha(V^-,s)-V^-\big).
\end{align}

Notice $F(0)=0$, and (using \eqref{lambda-const-manfld})
\begin{align}
F_{s^i}(s) = f_V\big(W^\alpha(V^-,s)\big)W^\alpha_{s^i}(W^\alpha(V^-,s))-\lambda^\alpha\big(W^\alpha(V^-,s)\big) W^{\alpha}_{s_i}(V^-,s)\big) \equiv 0. 
\end{align}

Now we check for entropy admissibility.  Define
\begin{align}
E(s) := q\big(W^\alpha(V^-,s)\big)-q(V^-)-\lambda^\alpha\big(W^\alpha(V^-,s)\big)\Big(e\big(W^\alpha(V^-,s)\big)-e(V^-)\Big). 
\end{align}

Then $E(0)=0$ and
\begin{align}
E_{s^i}(s) &= q_V\big(W^\alpha(V^-,s)\big)W_{s^i}^\alpha(V^-,s)-\lambda^\alpha\big(W^\alpha(V^-,s)\big)e_V\big(W^\alpha(V^-,s)\big)W_{s^i}^\alpha(V^-,s)  \\
&= e_V\big(W^\alpha(V^-,s)\big) \Big(f_V\big(W^\alpha(V^-,s)\big)-\lambda^\alpha\big(W^\alpha(V^-,s)\big)I\Big)W_{s^i}^\alpha(V^-,s) \equiv 0,
\end{align}

(using the property of entropy-entropy flux pairs).  Therefore, this choice of $V^\pm$ and $\xi$ satisfies \eqref{everywherepointavg} for either $x>0$ or $x<0$.   This means that any two states on a leaf of the foliation can be the left and right sides of a \emph{contact discontinuity} located at $\xi$, if $\xi$ is the (constant) value of $\lambda^\alpha$ on that leaf.  To that end, we will sometimes refer to $W^\alpha(V^-,s)$ as the \emph{contact manifold} through $V^-$.  

We now use a theorem from \cite{freistuhler1991linear} to prove that these are the only choices that satisfy \eqref{RH2} for $\epsilon$ sufficiently small.

\begin{theorem}[Freist\"{u}hler \cite{freistuhler1991linear}]\label{Freist} For any $(V^-, V^+, \xi)$ in a sufficiently small neighborhood of $(\Vb, \Vb, \lambda^{\alpha}(\Vb))$, if $(V^-, V^+, \xi)$ satisfy the Rankine-Hugoniot jump condition then $V^-$ and $V^+$ must lie on the same leaf of the characteristic foliation, and $\xi = \lambda^{\alpha}(V^-) = \lambda^{\alpha}(V^+)$.
\end{theorem}

We then have the following lemma.
\begin{lemma} \label{lem:contver}
For all $\xi \in I^\alpha$, $\xi \mapsto \lambda^\alpha\big(V(\xi)\big)$ is continuous.

In addition, if $\xi \neq \lambda^\alpha\big(V(\xi)\big)$ on some open interval in $I^\alpha$, then $V$ is constant on this interval.
\end{lemma}
\begin{proof}
Suppose $V$ is discontinuous at $\xi \in I^\alpha$.  Then we can find a pair of subsequences so that $[V] \neq 0$ and \eqref{RH2} is satisfied.  However, Theorem \ref{Freist} applies and so $V^\pm$ lie on the same leaf of the foliation, and so $[\lambda^\alpha(V)] = 0$.  This holds for any pair of subsequences, and so $\lambda\big(V(\xi)\big)$ is continuous at all $\xi \in I^\alpha$.  

Similarly, if $\xi \neq \lambda^\alpha\big(V(\xi)\big)$, then by Theorem \ref{Freist}, \eqref{RH2} cannot be satisfied for any pair of sequences unless $[V]=0$.  Therefore, $V$ itself must be continuous on such an open interval, and Theorem \ref{th:Vconst} shows it is constant.
\end{proof}
  
  \subsection{Intermediate State}\label{intstate}
  \begin{lemma}
  Let $W^{\alpha}(V^-,s)$ be as above.  For every $V^-, V \in \Peps$, there exists a unique 
  $s=s(V^-,V) \in B_\delta(0) \subset \R^{p_\alpha}$ such that 
  
  \begin{align}
  \hat{P}^\alpha \big(W^{\alpha}(V^-,s),V\big)\big(V-W^\alpha(V^-,s)\big) = 0, \label{zeroproj}
  \end{align}
  (for some $\delta > 0$).
  \end{lemma}

  \begin{proof}
Recall that $\hat{P}^{\alpha}$ has rank $p_\alpha$, so we can view the map $\bold{F}$ defined as
\begin{align}
\bold{F}(V^-,s,V) := \hat{P}^\alpha \big(W^{\alpha}(V^-,s),V\big)\big(V-W^\alpha(V^-,s)\big)
\end{align}
mapping $\R^{p_\alpha+2m}$ to $\R^{p_\alpha}$. (More concretely, we may define $\bold{\tilde F}$ to simply be the $p_\alpha$ entries of $\bold{F}$ corresponding to the $p_\alpha$ linearly independent rows of $\hat{P}^\alpha$, which do not change on $\Peps$ since linear independence is an open condition and $\epsilon$ can be decreased.)  Then, 
\begin{align}
\bold{F}_s(\Vb,0, \Vb) = -\hat{P}^{\alpha}(\Vb,\Vb)\Big( W^{\alpha}_{s^1} (\Vb, 0) \big| ... \big| W^\alpha_{s^{p_\alpha}}(\Vb, 0) \Big).
\end{align}
However, recall that each $W^{\alpha}_{s^i}(\Vb,0)$ is an eigenvector of $f_V(\Vb)$, and therefore lies in the total eigenspace (which at $(\Vb,\Vb)$ is just the eigenspace) that $\hat{P}^\alpha$ is projecting onto.  Therefore,
\begin{align}
\bold{F}_s(\Vb,0, \Vb) = - \Big( W^{\alpha}_{s^1} (\Vb, 0) \big| ... \big| W^\alpha_{s^{p_\alpha}}(\Vb, 0) \Big).
\end{align}
By construction of the contact manifold, this has rank $p_\alpha$, since the span of the columns is precisely $R^\alpha(\Vb)$.  By the implicit function theorem we obtain $s(V^-,V)$ close to the origin satisfying $\bold{F}(V^-,s,V) = 0$ for $\epsilon$ sufficiently small.
  \end{proof}
  
  \subsection{Regularity of $\beta$ Components}
  The following lemma establishes the regularity of the total projections along the other groups $\beta \neq \alpha$.
  
  \begin{lemma}
  There exists a set $E$ of full measure in $I^{\alpha}$ such that for all $\xi_0 \in E$, $\xi \in I^{\alpha}$, $\beta \neq \alpha$, and $i=1,...,p_\beta$ we have
  
  \begin{equation}
  \begin{split} 
  \hat{P}^\beta\Big(W^\alpha\big(V(\xi_0),s(\xi)\big),V(\xi)\Big)&\Big(V(\xi)-W^\alpha\big(V(\xi_0),s(\xi)\big)\Big) \\
  &= o(|\xi-\xi_0|) + \bO\big(|V(\xi)-V(\xi_0)|\cdot|\xi-\xi_0|\big),
  \end{split} \label{eq:littleo}
  \end{equation}
  provided that $\xi_0 = \lambda^\alpha\big(V(\xi_0)\big)$, and defining $s(\xi) := s\big(V(\xi_0),V(\xi)\big)$ in the context of the previous lemma.
  
  \end{lemma}

  \begin{proof}
  From \eqref{everywherepointavg}, we have
  
  \begin{align}
\Big(f\big(V(\xi)\big)-f\big(V(\xi_0)\big)\Big)-\xi_0 \big(V(\xi)-V(\xi_0)\big) &=   \int_{\xi_0}^{\xi} V(\xi)-V(\eta) d\eta. 
  \end{align}
  
  We claim that the left hand side is equal to
  
  \begin{align}
  \Big(f\big(V(\xi)\big)-f\big(W^\alpha\big(V(\xi_0),s(\xi)\big)\big)\Big)-\xi_0 \Big(V(\xi)-W^\alpha\big(V(\xi_0),s(\xi)\big)\Big). 
  \end{align}
  This is by design, since a contact at $\xi_0$ satisfies the Rankine-Hugoniot condition for the states $V(\xi_0)$ and $W^\alpha\big(V(\xi_0),s(\xi)\big)$.  For ease of reading, abbreviate $V(\xi_0) = V_0$, $V(\xi) = V$, and $W^\alpha\big(V^-,s(\xi)\big)$ as $W$.  Then
  \begin{align}
  \big(f(V)-f(V_0)\big) &- \xi_0 (V-V_0)  \\
  &= \big(f(V)-f(W)+f(W)-f(V_0)\big)-\xi_0(V - W + W - V_0)  \\
  &= \big(f(V)-f(W)\big) - \xi_0 ( V - W) + \big(f(W)-f(V_0)\big)-\xi_0(W-V_0)  \\
  &= \big(f(V)-f(W)\big) - \xi_0 ( V - W) + 0, 
  \end{align}
  since $\xi_0 = \lambda^{\alpha}(V_0)$ by assumption.  Then, using the averaged matrix we have
  \begin{align}
  \Big(\hat{A}(W,V)-\xi_0\Big)\big(V-W) = \int_{\xi_0}^{\xi} V-V(\eta) d\eta. 
  \end{align}
  Left multiply both sides by $\hat{P}^\beta(W,V)$ and add and subtract $V_0$ inside the integral on the right to obtain
  \begin{align}
 \hat{P}^\beta(W,V) &\Big(\hat{A}(W,V)-\xi_0\Big)\big(V-W) \\
  &=  \hat{P}^{\beta}(W,V) \left( \int_{\xi_0}^{\xi} V_0-V(\eta) d\eta + \big(V-V_0\big)(\xi-\xi_0) \right)
  \end{align}
  
  Lebesgue's differentiation theorem asserts that the integral on the right side is $o(|\xi-\xi_0|)$ for $\xi_0 \in E$, a set of full measure.  Using the properties of the total projection we obtain
    \begin{align}
\Big(\hat{A}(W,V)-\xi_0\Big)\hat{P}^\beta(W,V)(W-V) &=  o(|\xi-\xi_0|) + \bO\big(|V(\xi)-V(\xi_0)|\cdot|\xi-\xi_0|\big) 
  \end{align}
  However, since $\xi_0 \in I^\alpha$, and any eigenvalues in any $\beta$-group are thus uniformly bounded away from $\xi_0$, $(\hat{A}(W,V)-\xi_0)$ is uniformly non-degenerate on $\hat{P}^\beta \R^m$, and so
  \begin{align}
  \Big|\Big(\hat{A}(W,V)-\xi_0\Big)\hat{P}^\beta(W,V)(W-V)\Big| \geq   \delta \Big|\hat{P}^\beta(W,V)(W-V)\Big|,
  \end{align}
  for some $\delta > 0$, and the result follows.
  \end{proof}

  \subsection{Regularity of $V$ on Subsequences}
  The following lemma applies the main result of \cite{MR2921865} to obtain a subsequence on which $V(\xi)$ is Holder continuous, although for our purposes continuity would be enough.

  \begin{lemma}
  For almost every $\xi_0 \in E \subset I^\alpha$, there exists a subsequence $(\xi_n) \rightarrow \xi_0$ and $0 < C < \infty$ such that
  \begin{align}
  |V(\xi_n)-V(\xi_0)| \leq C|\xi_n-\xi_0|^{1/m}. \label{holder-m}
  \end{align}
  
  That is, for almost all $\xi_0 \in E \subset I^\alpha$ there exists a set $D := \{\xi_0\} \cup \left\{(\xi_n)\right\}$ such that $V|_D$ is Holder continuous with exponent $\frac{1}{m}$.
  \end{lemma}

  \begin{proof}
  As $V(\xi)$ is a function from $\R$ to $\R^m$, the result follows directly from \cite{MR2921865}.
  \end{proof}

  \begin{lemma}
  For $(\xi_n) \rightarrow \xi_0$ as above, we have the following estimate for all $\beta \neq \alpha$: 
  \begin{align}
 \hat{P}^\beta\Big(W^\alpha\big(V(\xi_0),s(\xi_n)\big),V(\xi_n)\Big)\Big(V(\xi_n)-W^\alpha\big(V(\xi_0),s(\xi_n)\big)\Big) &=  o(|\xi_n-\xi_0|). \label{beta-lo}
  \end{align}
  
  \end{lemma}

  \begin{proof}
  The previous lemma proves that $V|_D$ is continuous, and so $|V(\xi_n)-V(\xi_0)|$ is $o(1)$ and the result follows immediately from \eqref{eq:littleo}.\\
  \end{proof}

  \begin{lemma}
  In the setting of the previous lemmas, 
  \begin{align}
  \Big(V(\xi)-W^\alpha\big(V(\xi_0),s(\xi_n)\big)\Big) &=  o(|\xi_n-\xi_0|).
  \end{align}
  \end{lemma}
  \begin{proof}
  We have
  \begin{align}
   & \Big(V(\xi_n)-W^\alpha\big(V(\xi_0),s(\xi_n)\big)\Big) \\ &= \sum_{\beta}  \hat{P}^\beta\Big(W^\alpha\big(V(\xi_0),s(\xi_n)\big),V(\xi_n)\Big) \Big(V(\xi_n)-W^\alpha\big(V(\xi_0),s(\xi_n)\big)\Big) \\
    &=  \sum_{\beta\neq\alpha}  \hat{P}^\beta\Big(W^\alpha\big(V(\xi_0),s(\xi_n)\big),V(\xi_n)\Big) \Big(V(\xi_n)-W^\alpha\big(V(\xi_0),s(\xi_n)\big)\Big) \\ &=o(|\xi_n-\xi_0|),
  \end{align}
  due to the clever choice of $s(\xi_n)$.
  \end{proof}
  \subsection{Existence of at Most One Contact}
  We now combine the results of the previous three sections to obtain the main result.
  \begin{theorem}
    \label{th:lindeg}
    On a linearly degenerate sector, 
    $V$ is either constant, or constant on each side of a single contact discontinuity.
\end{theorem}
\begin{proof}
    By Lemma \ref{lem:contver},
    $F:=\{\xi\in I^\alpha~|~\xi=\lambda^\alpha(V(\xi))\}$ is closed and
    $V$ is constant on $I^\alpha\backslash F$.
    
    Assume there are $\xi_1,\xi_2\in F$ and $\eta\in I^\alpha \setminus F$ with $\xi_1<\eta<\xi_2$.
    Then we can choose a maximal $]\eta^-,\eta^+[$ containing $\eta$ but not meeting $F$.
    Necessarily $\eta^\pm\in F$, so $\eta^+=\lambda^\alpha(V(\eta^+))$ and $\eta^-=\lambda^\alpha(V(\eta^-))$.
    But $V$ is constant on $]\eta^-,\eta^+[$, so $\eta^+=\eta^-$, which is a contradiction.

    Hence $F$ must be a closed interval.

    Assume $F$ has positive length, and pick $\xi_0 \in F$ with $F \supset (\xi_n) \rightarrow \xi_0$ such that \eqref{holder-m} implies \eqref{beta-lo}.  Then,
    \begin{align}
    \lambda^\alpha\big(V(\xi_n)\big) &= \lambda^\alpha\Big(W^\alpha\big(V(\xi_0),s(\xi_n)\big)\Big) + \bO\Big(\Big|V(\xi_n)-W^\alpha\big(V(\xi_0),s(\xi_n)\big)\Big|\Big) \\
    &= \lambda^\alpha\big(V(\xi_0)\big) + \bO\Big(\Big|V(\xi_n)-W^\alpha\big(V(\xi_0),s(\xi_n)\big)\Big|\Big) \\
    &= \xi_0 + o(|\xi_n-\xi_0|).
    \end{align}
    However, since $\xi_n \in F$, this implies
    \begin{align}
    \xi_n - \xi_0 = o(|\xi_n-\xi_0|),
    \end{align}
    a contradiction.
    \newline Thus $F$ must be a point, or empty (which can be ruled out, but this is unnecessary).
\end{proof}

\section{Genuinely Nonlinear Sectors and Global SBV Regularity}
Outside linearly degenerate sectors, the results in \cite{elling2012steady} apply with the same proofs as in the strictly hyperbolic case, and so we have the following (see \cite{elling2012steady} for detailed discussions of shocks and simple waves).

\begin{theorem}\label{structure} Under Assumptions 1-5, and supposing that $f^x_U(\Ub)$ has only positive eigenvalues, we have:
\noindent (a) $V$ must be constant outside of $2m$ thin sectors centered at the roots of \eqref{det2}, which we can group as $m$ forward $(x > 0)$ and $m$ backward $(x<0)$ sectors.  \newline
(b) Linearly degenerate sectors each contain at most one contact discontinuity.\\
(c) Genuinely nonlinear forward sectors each contain at most one shock or simple wave.  \newline
(d) Genuinely nonlinear backward sectors can each contain infinitely many shocks and simple waves, but there cannot be consecutive simple waves.  \newline
(e) Each shock wave has a neighborhood on each side on which $U$ is constant, and the size of the neighborhood is lower bounded proportionally to shock strength. \\
(f) The width of the sectors and constant of proportionality depend only on the system and $\epsilon$.
\end{theorem}
\begin{rem}
If $f^x_U(\Ub)$ has only negative eigenvalues, then the ``forward sectors'' (named since they behave like \emph{forward-in-time} Riemann problem solutions, in which at most one wave from each family appears) lie in $x<0$ while the ``backward sectors'' lie in $x>0$.
\end{rem}
\begin{rem}
If $e$ is convex, we have shown that Assumptions 1-5 show that the solution in $x>0$ must be the unique Lax solution\footnote{Lax's solution is a forward-in-time self-similar solution for sufficiently close Riemann states consisting of up to $n$ waves interspersed with constant states, and is unique among solutions having this structure.  See \cite{serre1999systems}, Section 4.6 for details.} to the (one-spatial dimension $y$ with $x$ functioning as time) Riemann problem with data prescribed on the $y$-axis, and if $e$ is concave, then this is true in $x<0$.  Note that uniqueness is not expected in the half-plane corresponding to the backward-in-time solution, but our analysis shows the solution will still be of bounded variation.  Since \eqref{weakformV} are the same equations that a self-similar (function of $\xi = x/t$) solution to a one-dimensional system of conservation laws, our regularity results (and the uniqueness implied by our structural results and the uniqueness of Lax's Riemann solution ) apply to one-dimensional systems.  We have therefore extended uniqueness from the class of self-similar solutions consisting of up to $n$ waves with constant states in between to the class consisting of small self-similar $L^\infty$ perturbations of a constant state.  This extends an earlier result of Heibig \cite{heibig1990regularite}, which considered forward-in-time Riemann solutions to one-dimensional strictly hyperbolic \emph{genuinely nonlinear} systems, and so our analysis can treat systems that can have repeated linearly degenerate eigenvalues.
\end{rem}
\begin{proof}
The statement regarding degenerate sectors follows from Theorem \ref{th:lindeg}, and the remaining statements follow directly from \cite{elling2012steady}.
\end{proof}
In addition, the regularity results contained in \cite{elling2012steady} also apply, now that we have at most one contact in degenerate sectors.
\begin{theorem}
Suppose $V$ satisfies the hypotheses of the previous theorem.  Then $V$ is of bounded variation.  More specifically, $V = V_S + V_L$ where $V_S$ is a saltus (jump) function of bounded variation, and $V_L$ is Lipschitz, and so $V$ is a special function of bounded variation.  The total variation of $V_S$ and the Lipschitz constant of $V_L$ are independent of $V$ and only depend on the system and $\ueps.$
\end{theorem}
\begin{proof}
Again, now that linearly degenerate sectors have been treated by Theorem \ref{th:lindeg}, and the analysis in \cite{elling2012steady} works the same everywhere else, the claim follows.
\end{proof}

\section{Comparison to the Strictly Hyperbolic Case}\label{stricthyp}
Though much of the analysis is similar to the strictly hyperbolic case treated in \cite{elling2012steady}, there are some crucial differences.  

As discussed in Section \ref{matravg}, the construction of the averaged matrix $\hat{A}(V^\pm)$ is less straightforward in this non-strictly hyperbolic case and relies on convexity of $e(\cdot)=\psi^x\big((f^x)^{-1}(\cdot))$.  In \cite{elling2012steady}, this assumption was not needed to construct the averaged matrix (the more straightforward definition \eqref{naive} could be used), it was only necessary that
\begin{equation}
e_{VV}r^\alpha r^\alpha \neq 0
\end{equation}
for each $\alpha$ corresponding to a genuinely nonlinear field.  This quantity cannot vanish as long as $f^x_U$ is non-degenerate, but if $f^x_U$ has eigenvalues with different signs then this quantity will have different signs for different fields.  This was needed to identify which part of the Hugoniot locus corresponded to admissible shocks for genuinely nonlinear fields, and this ultimately determined whether the ``forward'' and ``backward'' sectors corresponded to $x>0$ or $x<0$.  Therefore, there are some strictly hyperbolic systems and background states that can be treated by \cite{elling2012steady} for which neither $x>0$ or $x<0$ (or any half-plane) is analogous to a forward-in-time Riemann problem solution to a one-dimensional system of conservation laws (recall that forward-in-time Riemann problem solutions consist of at most one wave from each family --- which is the behavior described in Theorem \ref{structure} for forward sectors).  However, since in the non-strictly hyperbolic case we need the eigenvalues of $f^x_U(\Ub)$ to all have the same sign in order to construct $\hat{A}$, the systems and background states treated in this paper have the property that either $x>0$ or $x<0$ will correspond to a forward-in-time Riemann problem solution.    

More importantly, the entire approach for linearly degenerate sectors is completely different in this non-strictly hyperbolic case.  In the strictly hyperbolic case in \cite{elling2012steady}, the situation was similar to that in Theorem \ref{th:lindeg}:
\begin{align}
\lambda^\alpha\big(V(\xi)\big) &= \xi \mbox{ on a closed interval } F \subset I^\alpha, \mbox{ and } \label{resonant} \\
\hat{P}^\beta\big(V(\xi_0),V(\xi)\big) \big(V(\xi)-V(\xi_0)\big) &= \bO(|\xi-\xi_0|) \mbox{ for all } \beta \neq \alpha. \label{blip}
\end{align}
The approach was then to use a result due to Saks \cite{saks1937theory}: for \emph{any} function $F: \R \supset A \rightarrow \R$, the set of points $x$ at which 
\begin{equation}
\liminf_{h \rightarrow 0+}\frac{|F(x+h)-F(x)|}{h} =  \infty
\end{equation}
is of measure zero.   This implies differentiability at almost every $\xb \in F$ after restriction to a subsequence converging to $\xi_0$ of the real valued function
\begin{equation}
\xi \mapsto l^\alpha(\Vb)\big(V(\xi)-V(\xi_0)).
\end{equation}
Taking further and further subsequences for the $m-1$ other (Lipschitz from \eqref{blip}) components allowed for the construction of a sequence $(\xi_n) \rightarrow \xb$ so that
\begin{equation}
\lim_{n \rightarrow \infty} \frac{V(\xi_n)-V(\xb)}{\xi_n - \xb} =: V'
\end{equation}
exists and is finite.  It was then shown that
\begin{equation}
\Big(f_V\big(V(\xb)\big)-\xb I \Big) V' = 0,
\end{equation}
implying $V' \parallel r^\alpha\big(V(\xb)\big)$, contradicting
\begin{equation}
\lambda_V\big(V(\xb)\big) V' = 1,
\end{equation}
(which follows from \eqref{resonant}), due to linear degeneracy.  

However, for a linearly degenerate eigenvalue of multiplicity $p_\alpha > 1$, \eqref{blip} only guarantees $m-p_\alpha$ Lipschitz components, and there are $p_\alpha$ components that we have no regularity information about (indeed at a contact with fixed $V^-$, there is a $p_\alpha$-dimensional submanifold of allowable $V^+$ to which a jump discontinuity can occur).  Therefore, for the previous argument to work we need differentiability after restriction of a function taking values in $\R^{p_\alpha}$.  Saks's theorem does not apply, and functions resembling space filling curves show explicit counterexamples.  In \cite{MR2921865}, the first author proved an analogous result --- for \emph{any} function $F: \R^k \supset A \rightarrow \R^p$, the set of points $x$ at which
\begin{equation}
\liminf_{|h|\rightarrow 0} \frac{|F(x+h)-F(x)|}{|h|^{k/p}} = \infty
\end{equation}
is of measure zero.  This paper also presents an explicit example demonstrating the sharpness of the exponent $k/p$, and so we cannot hope to obtain differentiability after restriction for an arbitrary function $F: \R \supset A \rightarrow \R^{p_\alpha}$ if $p_\alpha > 1$.  Therefore, the different approach of constructing the ``intermediate state'' in Section \ref{intstate} is necessary.

\appendix

\section{Regularity of Eigenvalues and Eigenvectors} \label{evaluesmooth}
The standard implicit function theorem argument for smoothness of simple eigenvalues does not work if there are repeated eigenvalues.  (We will use the convention that when the eigenvalues are indexed with a subscript they are repeated according to their multiplicity, where superscript indices only label the distinct eigenvalues.)  It is well known (see \cite{serre2010matrices}) that the unordered set of eigenvalues (repeated according to multiplicity) of an $m \times m$ matrix is a continuous function of the matrix entries, with the spectrum being an element of $\C^m \setminus \sim$, where
$\left\{ \lambda_\alpha \right\}_{\alpha=1}^m \sim \left\{ \mu_\alpha \right\}_{\alpha=1}^m$ if $\lambda_\alpha = \mu_{\sigma(\alpha)}$ for all $\alpha=1..m$ and some $\sigma \in S_n$.  The metric on this quotient space is given by
\begin{align}
d\left( \left\{\lambda_\alpha \right\}, \left\{ \mu_\alpha \right\} \right) = \min_{\sigma \in S_n} \max_{1 \leq \alpha \leq m} | \mu_\alpha - \lambda_{\sigma(\alpha)} |.
\end{align}
However, if the matrices in question are continuous functions of say $z \in D \subset \R^k$ such that the eigenvalues are real for all $z \in D$, then it is clear we can label

\begin{align}
\lambda_1(z) \leq \lambda_2(z) \leq ... \leq \lambda_m(z)
\end{align}
such that $\lambda_\alpha(z)$ is a continuous function of $z$ for $\alpha= 1, .., m$.

The smoothness of the eigenvalues and eigenvectors is more delicate when the eigenvalues are not simple --- in fact there are many counterexamples.  However, if the matrices $A(z)$ have the property that each distinct $\lambda^\alpha(z)$ has constant algebraic multiplicity $p_\alpha$ and $p_\alpha$ linearly independent eigenvectors for all $z \in D$, then around each $z_0 \in D$ there exists a neighborhood $D_{z_0} \ni z_0$ such that, for $\alpha,\beta=1,..,n, i = 1,..,p_\alpha,j=1,..,p_\beta$,
\begin{align}
&\lambda^\alpha(z) : D \rightarrow \R \\
&r^{\alpha,i}(z): D_{z_0} \rightarrow \R^m \\
&l^{\alpha,i}(z): D_{z_0} \rightarrow \R^m
\end{align}
are smooth functions satisfying for all $z \in D_{z_0}$
\begin{align}
A(z)r^{\alpha,i}(z) &= \lambda^{\alpha}(z) r^{\alpha,i}(z), \\
l^{\alpha,i}(z) A(z) &= l^{\alpha,i}(z) \lambda^{\alpha}(z), \\
l^{\alpha,i}(z)r^{\beta,j}(z)&=\delta_{\alpha \beta}\delta_{ij}, \label{normalization} \\
|r^{\alpha,i}(z)| & = 1.
\end{align}
Moreover, the set of right (and left) eigenvectors is linearly independent for all $z$, and for each given family the right eigenvectors can be taken to be orthonormal.  A proof of these statements for a single semisimple eigenvalue of constant multiplicity can be found in \cite{nomizu1973characteristic}.
\begin{theorem}[Nomizu \cite{nomizu1973characteristic}] Let $D \subset \R^k$ be open and $A: D \rightarrow \mathbb{M}_{m}(\R)$ be a smooth mapping such that $A(z)$ is diagonalizable for all $z$.  If $\lambda$ is a continuous function on $D$ such that for every $z \in D$ the value $\lambda(z)$ is an eigenvalue of $A(z)$ with the common multiplicity $p$, then $\lambda$ is smooth.  Furthermore, for each $z_0$, there exist smooth eigenvectors $r^1, ..., r^p$ of a neighborhood $D_{z_0}$ of $z_0$ into $\R^m$ such that, for each $z \in D_{z_0}$, $r^1(z), ..., r^p(z)$ form an orthonormal basis of the eigenspace of $A(z)$ for  $\lambda(z)$.
\end{theorem}
Apply this theorem to each eigenvalue, and take the left eigenvectors to be the rows of the inverse matrix of the matrix of right eigenvectors to obtain the normalization \eqref{normalization}.

 \section{Calculations for Full Euler}\label{eul}
The full Euler equations are given by
\begin{align}
f^x(U) = \left( \begin{array}{c} m \\ m^2\rho^{-1} + p \\ mn\rho^{-1} \\ \frac{1}{2}m(m^2+n^2)\rho^{-2}+me+mp\rho^{-1} \end{array} \right), f^y(U) = \left( \begin{array}{c} n \\ mn\rho^{-1} \\ n^2\rho^{-1}+ p \\ \frac{1}{2}n(m^2+n^2)\rho^{-2}+ne+np\rho^{-1} \end{array} \right),
\end{align}
where $\rho$ is the density, $m$ and $n$ are the horizontal and vertical momentum densities, $p$ is the pressure, and $e$ is the specific internal energy.  The conserved quantities are $\rho$, $m$, $n$, and the total energy per unit volume $\rho E := \frac{1}{2\rho}(m^2+n^2)+\rho e$.  For convenience, we can calculate the eigenvalues and check for genuine nonlinearity/linear degeneracy by considering the matrix pair $(f^x_W,f^y_W)$ with convenient variables $W = (\rho, m, n, S)$ where $S$ is the specific entropy (the invariance of eigenvalues and nonlinearity properties was discussed at the end of Section  \ref{gennonlindeg}).  It is well known (see \cite{MR0029615}) that of the five thermodynamic quantities $\rho$, $e$, $S$, $p$ and temperature $T$, only two are independent, so once two have been chosen the other three can be expressed in terms of those two.  We have that
\begin{align}
p_\rho(\rho, S) &=: c^2 > 0, \\
e_\rho(\rho,S) &= \frac{p}{\rho^2},\\
e_S(\rho,S) &= T,
\end{align}
where $c$ is the \emph{sound speed}.  The derivatives of the fluxes in terms of the variables $W$ are then
\begin{align}
f^x_W &= \left( \begin{array}{cccc} 0 & 1 & 0 & 0 \\ -m^2\rho^{-2}+c^2 & 2m\rho^{-1} & 0 & p_S \\ -mn\rho^{-2} & n\rho^{-1} & m\rho^{-1} &0 \\ -m(m^2+n^2)\rho^{-3}+mc^2\rho^{-1} & \frac{3}{2}m^2\rho^{-2}+\frac{1}{2}n^2\rho^{-2}+e+p\rho^{-1} & mn\rho^{-2} & mT + mp_S\rho^{-1} \end{array} \right), \\
f^y_W &= \left( \begin{array}{cccc} 0 & 0 & 1 & 0 \\ -mn\rho^{-2} & n\rho^{-1} & m\rho^{-1} & 0 \\ -n^2 \rho^{-2} + c^2 & 0 & 2n\rho^{-1} & p_S \\ -n(m^2+n^2)\rho^{-3} + nc^2\rho^{-1} & mn\rho^{-2} & \frac{3}{2} n^2 \rho^{-2}+\frac{1}{2}m^2\rho^{-2}+e+p\rho^{-1} & nT + np_S\rho^{-1} \end{array} \right).
\end{align}
Using Maple, we can compute the eigenvalues and eigenvectors.   There are two simple eigenvalues corresponding to acoustic waves:
\begin{align}
\lambda^\pm(U) = \frac{mn \pm \rho c \sqrt{m^2 + n^2 - (\rho c)^2}}{m^2-(\rho c)^2},
\end{align}
and a double eigenvalue corresponding to shear waves and entropy jumps:
\begin{equation}
\lambda^0(U) = \frac{n}{m}.
\end{equation}

The background state we consider will be of the form $\Ub = (1, M_0, 0, E_0)$, where the units have been scaled so that the background density and sound speed are 1, and so the background \emph{Mach Number} $|M_0|$ is the absolute value of the background horizontal momentum.  So that these eigenvalues are real and distinct, we must assume $|M_0| > 1$, which is simply requiring the background state to be \emph{supersonic}, and it is clear that the eigenvalues and eigenvectors are smooth functions on a neighborhood of this $\Ub$.  The eigenvectors considered below are also smooth on a neighborhood of $\Ub$.

Two linearly independent eigenvectors for $\lambda^0$ are
\begin{align}
r^{0,1}(U) = \left(\begin{array}{c} 0 \\ m \\ n \\ 0 \end{array} \right), \qquad r^{0,2}(U) = \left( \begin{array}{c} -p_S \\ 0 \\ 0 \\ c^2 \end{array}\right).
\end{align}
As $\lambda^0$ is semisimple of constant multiplicity, Boillat's theorem (Theorem \ref{boillat}) guarantees that $\lambda^0$ is linearly degenerate, though an easy calculation confirms this as well.  For any contact discontinuity corresponding to $\lambda^0$, the flow will be tangential to the discontinuity.  The $r^{0,1}$ field corresponds to a \emph{shear wave}, which is a discontinuity in the tangential velocity, and the $r^{0,2}$ field corresponds to an \emph{entropy wave}, in which the entropy and density are different but in a manner so that the pressure is constant on either side.  

The acoustic characteristic fields $\lambda^\pm$ are genuinely nonlinear --- to confirm this it suffices to check at the background state with $n=0$ since genuine nonlinearity is an open condition.  From Maple we have that
\begin{align}
r^\pm(\Ub) = \left( \begin{array}{c}  \pm M_0  \\
        \pm (M_0^2-1) \\
        \sqrt{M_0^2-1} \\ 0
    \end{array}\right). 
\end{align}
However, the first three entries are precisely the $r^\pm$ eigenvectors for the isentropic case investigated in \cite{elling2012steady}, $\lambda^\pm$ are identical in the full and isentropic cases, and the first three $W$ variables we consider in this case are precisely the conserved quantities for the isentropic case.  Since $\lambda^\pm$ for full Euler are independent of $S$, it is clear that in our context of full Euler $\lambda^\pm_W r^\pm$ is identical to $\lambda^\pm_U r^\pm$ from the isentropic case in Section 18 of \cite{elling2012steady}.  Genuinely nonlinearity was demonstrated for $c_\rho > -1$, which is true for commonly used equations of state (in particular, any polytropic gas with $\gamma > 1$.)

As discussed in Chapter II, Section 1.1 of \cite{godlewski1996numerical}, taking $\eta(U) = -\rho S$, $\psi^x(U) = -m S$, $\psi^y(U) = -n S$  yields an entropy-entropy flux pair.  The second law of thermodynamics implies that $S$ is a strictly concave function of $\rho^{-1}$ and $e$, which \cite{godlewski1996numerical} proves is equivalent to $-\rho S$ being a strictly convex function of $\rho, m, n,$ and $\rho E$.

Finally, in Chapter II, Section 2 of \cite{godlewski1996numerical}, it is shown that the eigenvalues of $f^x_U$ are
\begin{align}
\frac{m}{\rho} \pm c, \frac{m}{\rho}, \frac{m}{\rho}.
\end{align}
(A change of coordinates was introduced that preserves the eigenvalues, since the actual expression for $f^x_U$ is complicated due to the need of differentiating with respect to $\frac{m^2+n^2}{2\rho}+\rho e$.)  Therefore, if $M_0 > 1$, $e_{VV}$ is positive definite (by Lemma \ref{econvex}) and the system is hyperbolic in the positive $x$-direction and forward sectors lie in $x>0$.  If $M_0 < -1$, $e_{VV}$ is negative definite and the system is hyperbolic in the negative $x$-direction.

\bibliographystyle{chicago}
\bibliography{biblio}

\end{document}